\documentclass[12pt,a4paper]{amsart}
  {}
  {}
  {}
\pagespan{1}{}

\theoremstyle{plain}
\newtheorem{theorem}{Theorem}[section]
\newtheorem{proposition}[theorem]{Proposition}
\newtheorem{lemma}[theorem]{Lemma}
\newtheorem{corollary}[theorem]{Corollary}
\theoremstyle{definition}
\newtheorem*{definition}{Definition}
\theoremstyle{example}
\newtheorem{example}[theorem]{Example}

 \newcommand{\A}{\mathcal{A}}
 \newcommand{\p}{P_n}
\newcommand{\ph}{\hat{P}_n}
\newcommand{\nm}{\Bbb{N}}
 \newcommand{\U}{\mathcal{U}}
\newcommand{\X}{\mathcal{X}}

\begin{document}

\title[Contractibility of ultrapower of Fr\'echet
algebras ]
 {Contractibility of ultrapower of Fr\'echet
algebras }

\author{E. Feizi and J. Soleymani}
 \address{ Mathematics Department, Bu-Ali Sina University, 65174-4161, Hamadan,
 Iran}\email{efeizi@basu.ac.ir; j.soleymani@basu.ac.ir}


\begin{abstract}
The aim of this article is to study a number of relationships
between Fr\'echet algebra $\A$ and its ultrapower $({\A})_{\U}$.
We give a characterization in some aspects such as locally bounded
approximate identity. We consider the notion of contractibility
of ultrapower of Fr\'echet algebra, and  we show that if
$({\A})_{\U}$ has approximation property with good ultrafilter
${\U}$ then $\A$ is contractible if and only if $({\A})_{\U}$ is
contractible.
\end{abstract}
\maketitle

\section*{{Introduction}}

A Banach algebra $\A$ is said to be contractible if for every
Banach $\A$-bimodule $\X$ each continuous derivation from $\A$
into $\X$ is inner.
 Contractibility of Banach algebras has been described  in term of  diagonal. In this manner a Banach algebra is contractible if
and only if it has a  diagonal. For more details see \cite{Run}.

The theory of amenability and inclusions has recently been
developed in the direction of problems on nonnormed algebras.
Fr\'{e}chet  algebras, i.e, metrizable locally convex complete
topological algebras, are of particular interest. There are
essential differences between topologies of Banach and Fr\'echet
algebras. For example, every bounded subset of  Fr\'echet
nonnormed algebras has empty interior.

 Ultrapower of locally convex spaces and Banach spaces was introduced by
  Heinrich in \cite{Hen1},
  \cite{Hen2}, and \cite{Hen3}. Daws in \cite{Daw}
  discussed the relation between amenability of a Banach algebra and its ultrapower.

In this paper, we study the construction of  the ultrapower of
Fr\'echet algebras and  consider the contractibility of such
structures. For this purpose, we first show that for a Fr\'echet
algebra $\A$, and a countable incomplete ultrafilter $\U$, its
ultrapower $(\A)_{\U}$ is Fr\'echet algebra. Then, we define
approximation property for locally convex spaces and locally
diagonal for Fr\'echet algebras and  prove that for a good
ultrafilter $\U$, if $(\A)_{\U}$ has approximation property then
$\A$ has (locally) diagonal if and only if $(\A)_{\U}$ has
(locally) diagonal. The conclusion is that if $({\A})_{\U}$ has
approximation property with a good ultrafilter ${\U}$ then $\A$ is
contractible if and only if $({\A})_{\U}$ is contractible.

Before giving the proof of the above concepts, we define
 locally bounded approximate identity and show that a Fr\'echet algebra $\A$ is unital if and only
 if $(\A)_{\U}$ is so. Finally this applies to bounded approximate identity and locally bounded approximate identity.

\section{{Preliminaries}}
In this section some notations and preliminaries on Fr\'echet
algebras and ultrafilters will be setup. We generally follow
\cite{Bir, Kes, Dal, Mal} for more details.

 A complete topological algebra $\A$ whose topology is given by
an increasing countable family of sub-multiplicative semi-norms
$(P_n)$ is called a Fr\'echet algebra and is denoted by
$({\A},(P_n))$. By \cite[Corollary 3.2]{Mal} every Fr\'echet
algebra within a topological algebraic isomorphism is the
projective limit of a decreasing sequence of Banach algebras
which, for Fr\'echet algebra $({\A},(P_n))$ is denoted by
$\A=\displaystyle{\lim_{\longleftarrow}} {\A}_n $ where for $n\in
\Bbb{N}$, ${\A}_n$ is the completion of quotient normed  space
$\frac{\A}{P^{-1}_n\{0\}}$.

Let $\A$ be a Fr\'echet algebra. Then a Fr\'echet space $X$ is a
Fr\'echet $\A$-bimodule if  both sides module actions of $\A$ on
$X$ are continuous. For example, $\A$ itself is a Fr\'echet
$\A$-bimodule. Furthermore $X^*$, the dual space of a Fr\'echet
$A$-bimodule $X$ with the strong topology, is a locally convex
$\A$-bimodule with respect to the module operations defined by
\begin{eqnarray*}
\langle a\cdot f, x\rangle=\langle f, x\cdot a\rangle,
\quad\quad\: \langle f\cdot a, x\rangle=\langle f, a\cdot
x\rangle \quad\quad (x\in X),
\end{eqnarray*}
where $a\in \A$ and $f\in X^*$. In this case we say that $X^*$ is
the dual module of $X$.


A filter on a set $I$ is a subset $F$ of $P(I)$ such that: (i)
empty set is not in $F$; (ii) $F$  is closed under finite
intersection; (iii) $F$ is upper set. By Zorn's Lemma, the maximal
filters exist and are called ultrafilters. If $x \in X$ then the
collection $U_x = \{A \subseteq X : x \in A\}$ is an ultrafilter.
 This ultrafilter being the principal ultrafilter at x. Ultrafilters
which are not principal are called non-principal. We want to
study the concept of a limit along a filter. Let $F$ be a filter
on a set $I$ and $(x_i)_{i\in I}$ be a family in a topological
space. We write $x = \lim_{i\in F} x_i$ if for each open
neighborhood $U$ of $x$, then  $\{i \in I : x_i \in U\}\in F$. We
call an ultrafilter $\U$ countably incomplete when there exists a
sequence $(U_n)_{n\in \Bbb{N}}$ in $\U$ such that $U_1\supseteq
U_2\supseteq U_3\cdots$ and  $\bigcap_{n\in
\Bbb{N}}U_n=\emptyset$. In this paper, it is supposed that $\U$ is
countably incomplete ultrafilter.(For more details, see
\cite[1.3]{3} and \cite {Sim}).
\begin{proposition}\cite[Proposition 1.3.4]{3}\label{fil1}
 Let $X$ be a compact topological space, let $\U$
be an ultrafilter on a set $I$, and let $(x_i)_{i\in I}$ be a
family in $X$. Then there exists $x \in X$ with $x =\lim _{i\in
\U} x_i$. Furthermore, if X is Hausdorff, then $x$ is unique.
\end{proposition}
Note that the converse of the above Proposition is also true.

The rest of this section is devoted to the introducing of a new
kind of ultrafilter, which is called "good ultrafilter".
Keisler in 1964,  introduced the notion of an $\aleph^+$-good
ultrafilter in the field of sets (in particular, on
$\mathcal{P}(\aleph)$). He also showed the existence of
$\aleph^+$-good countably incomplete ultrafilters on
$\mathcal{P}(\aleph)$, where $\aleph$ is a cardinal and
$\aleph^+=2^{\aleph}$.


 Before
we give the definition of good ultrafilter, we need some notations
on functions. Let $I$ be a nonempty set, $P_f(I)$ the family of
finite subsets of $I$, and $\beta$ a cardinal. Suppose that $f,g$
are functions from the set $P_f(\beta)$ into the power set $P(I)$.
We say that $g\leq f$ whenever for all $u\in P_f(\beta)$,
$g(u)\subset f(u)$.
 Also, $f$ is monotonic whenever for all $u,w\in P_f(\beta)$
with $u\subset w$, $g(u)\supset g(w)$.
 The function $g$ is said to be additive whenever
  $$g(u\cup w)=g(u)\cap g(w), \quad\quad u,w\in
P_f(\beta).$$
By  \cite[Lemma 6.1.3]{Kes} every additive function is monotonic.
\begin{definition}
Let $\aleph$ be an infinite cardinal. An ultrafilter $\U$ on a set
$I$ is called $\aleph$-good ultrafilter
 if for each cardinal $\beta\le \aleph$ and every monotonic function $f$ from $P_f(\beta)$
 into $\U$ then  there exists an additive function $g$ from $P_f(\beta)$
 into $\U$ such that $g\leq f$.
\end{definition}

The next result guarantees that there are sufficiently many good
ultrafilters.

\begin{proposition}\cite[Theorem 6.1.4]{Kes}
Let $\aleph$ be an infinite cardinal. On each set of cardinality
not smaller than $\aleph$ there exists an $\aleph^+$-good
countably incomplete ultrafilter.
\end{proposition}

For a locally convex space $E$ let $\aleph_U(E)$ be the least
cardinality of a basis of neighborhoods of zero, $\aleph_B(E)$ be
the least cardinality of a fundamental (i.e. cofinal with respect
to inclusion) family of bounded sets and $\aleph(E) = \max
(\aleph_U(E), \aleph_B(E))$. Since by \cite[Corollary 16.3 ]{Kak}
every metrizable locally convex space that has a fundamental
sequence of bounded sets is normable,  for Fr\'echet space $E$
with $\aleph (E)=\aleph_0$,  $E$ is Banach space, where
$\aleph_0$ is the first infinite cardinal number.

 \section{Ultrapower of Fr\'echet algebra}

 Let $I$ be an index
set that equipped with a countably incomplete ultrafilter $\U$.
The linear subspace $\ell_{\infty}(I,\A)$ of $\prod_i{\A}$ for
Fr\'echet algebra $(\A,(\p))$ can be define  as follows:
$$\ell_{\infty}(I,{\A}):=\{(x_i)\in \prod_i{\A}:{\sup}_i{\p}(x_i)<\infty,\mbox{for
all }n\in{\nm} \}.$$

 This space with respect to semi-norms $Q_n$
which are defined as
$$Q_n((x_i)):=\sup_{i\in I}{\p}(x_i), \quad\quad (x_i)\in
\ell_{\infty}(I,{\A}),$$ is a Fr\'echet space. Since for any
$(x_i)\in \ell_{\infty}(I,{\A})$, the set $\{{\p}(x_i):i\in I\}$
is bounded in $\Bbb{R}$, the Proposition ~\ref{fil1} implies that
$\lim_{\U}{\p}(x_i)$ exists and closed subspace $\mathcal{N}_{\U}$ can be defined as follows:
$$\mathcal{N}_{\U}=\{(x_i)\in
\ell_{\infty}(I,{\A}):\lim_{\U}{\p}(x_i)=0 \mbox{ for all } n \in
\nm \}.$$
So the quotient space
$(\A)_{\U}:=\frac{\ell_{\infty}(I,{\A})}{\mathcal{N}_{\U}}$ is
well defined and its elements are denoted by
$(x_i)_{\U}:=(x_i)+\mathcal{N}_{\U}$ for $(x_i)\in
\ell_{\infty}(I,{\A})$.

Using above notations for all $n\in \Bbb{N}$ we can define
increasing semi-norms $Q'_n$ as follows on $({\A})_{\U}$:
\begin{equation}\label{usd}
Q'_n((x_i)_{\U}):=\inf\{Q_n((x_i)+(y_i)):(y_i)\in
\mathcal{N}_{\U}\}.
\end{equation}
Given semi-norms on ${(\A)}_{\U}$ make this space complete
according to  \cite[Proposition 3.2.7] {Sim},  so $({\A})_{\U}$ is
Fr\'echet space. As $\A$ is algebra, for all
$(x_i)_{\U},(y_i)_{\U}\in ({\A})_{\U}$ the product is defined by
$(x_i)_{\U}\cdot (y_i)_{\U}:=(x_iy_i)_{\U}$ and ${(\A)}_{\U}$
becomes Fr\'echet algebra.

\begin{proposition}\label{2.2}
Let $\A$ be a Fr\'echet algebra. Then for every $(x_i)_{\U}\in
{(\A)}_{\U}$
\begin{equation}\label{us}
Q'_n((x_i)_{\U})=\lim_{\U}\p(x_i).
\end{equation}
\end{proposition}

\begin{proof} Similar  to the proof of \cite[Theorem 2.4] {Kan}.
\end{proof}

We now develop a new version of bounded approximate identity in
Fr\'echet algebra which is called locally bounded approximate
identity.
\begin{definition}
Let $({\A},(P_n))$ be a Fr\'echet algebra. Then $\A$ has a locally
bounded approximate identity if for all $n\in \Bbb{N}$ there
exists a net $(u_i)_{i\in I}$ and a constant $C_n>0$ such that
for all $i\in I$, $P_n(u_i)<C_n$ and for all $a\in {\A}$,
$P_n(a-u_ia)\rightarrow 0$ and $P_n(a-au_i)\rightarrow 0$.
Moreover, it is assumed  that $\A$ has a bounded approximate
identity if there exists a net $(u_i)$ and a constant $C>0$ such
that  $P_n(u_i)<C$, $P_n(a-u_ia)\rightarrow 0$ and
$P_n(a-au_i)\rightarrow 0$, for all $i\in I, n\in {\Bbb N}, a\in
{\A}$.
\end{definition}
 Similarly left and right (locally) bounded approximate identity will be defined.

\begin{definition}
Let $({\A},(P_n))$ be a Fr\'echet algebra. Then $\A$ has a locally
bounded approximate unit if for all $n\in \Bbb{N}$, $a\in \A$ and
$\epsilon>0$, there exists a constant $C_n>0$ and $u\in \A$ such
that $P_n(u)<C_n$, $P_n(a-ua)<\epsilon$ and $P_n(a-au)<\epsilon$.
 $\A$ has a bounded approximate unit if for all $a\in \A$ there
exists $u\in \A$ such that for all $n\in \Bbb{N}$, and
$\epsilon>0$ there exists a constant $C>0$ such that $P_n(u)<C$,
$P_n(a-ua)<\epsilon$ and $P_n(a-au)<\epsilon$.
\end{definition}

 Now we investigate some relation between unit and (locally) bounded approximate unit of $\A$ and $({\A})_{\U}$
 the general structure of our proof is similar to the proof that is used by Daws in \cite {Daw},
  the only difference being that  we applied the semi-norms instead of the norms.
\begin{proposition}\label{3.1}
For a Fr\'echet algebra $\A$, $({\A})_{\U}$ is unital if and only
if $\A$ is unital.
\end{proposition}

\begin{proof} Because for a unital Fr\'echet algebra
with unity $e$ we can assume $\p(e)=1$ for all $n\in {\Bbb N}$,
(see \cite {Gol}).
$({\A})_{\U}$ is clearly unital.\\
Now let $e=(e_i)_{\U}\in ({\A})_{\U}$ be unit for $({\A})_{\U}$.
Then for all $\epsilon>0$, $i$ and $m\in {\Bbb N}$ there exist
$a_i\in{\A}$ such that for all $n\in \Bbb{N}$, $\p(a_i)\leq 1$ and
$$P_m(a_i-e_ia_i)\geq\sup\{\p(a-e_ia):a\in {\A},\p(a)\leq 1\mbox{ for
all n}\in \Bbb{N}\}-\epsilon.$$
 Let $b=(a_i)\in ({\A})_{\U}$ then
$b=eb=(e_ia_i)\in ({\A})_{\U}$ and
$$0=\lim_{\U}P_m(a_i-e_ia_i)\geq\lim_{\U}\sup\{\p(a-e_ia):a\in {\A},\p(a)\leq 1\mbox{ for
all n}\in \Bbb{N}\}-\epsilon.$$

As $\epsilon>0$ was arbitrary we have
$$\lim_{\U}\sup\{\p(a-e_ia):a\in {\A},\p(a)\leq 1\mbox{ for
all n}\in \Bbb{N}\}=0.$$

With the same method we have;
$$\lim_{\U}\sup\{\p(a-ae_i):a\in {\A},\p(a)\leq 1\mbox{ for
all n}\in \Bbb{N}\}=0.$$

This implies that for all $\epsilon>0$, the set
$U=\{i:\p(a-ae_i)+\p(a-e_ia)<\epsilon,a\in{\A},\p(a)\leq 1 \mbox{
for all }n\in\Bbb{N}\}\in {\U}$, then for all $i,j\in {\U}$ and
$n\in \Bbb{N}$, we have
$$\p(e_i-e_j)\leq \p(e_i-e_ie_j)+\p(e_j-e_ie_j)<2\epsilon.$$
Now we can extract a cauchy sequence from the family $(e_i)$
converges to say $e_{\A}\in {\A}$. It follows that $e_{\A}$ is
unit for $\A$, as required.
\end{proof}
\begin{lemma}
Let $({\A},(P_n))$ be a Fr\'echet algebra and let $m\geq 1$ and
$C>0$. Suppose that for each $a_1,\cdots,a_m\in \A$ and
$\epsilon>0$ there exists $u\in \A$ such that
$$P_n(a_i-ua_i)<\epsilon, \quad\quad P_n(u)<C,$$
where $n\in {\Bbb N}$ and $1\leq i\leq m$. Then $\A$ has a
bounded left approximate identity.
\end{lemma}
\begin{proof}
The proof is similar to the proof of \cite[Proposition 2.9.14] {Dal}.
\end{proof}
The above lemma also holds for bounded (right) approximate identity.
\begin{corollary}
A Fr\'echet algebra has bounded (right,left) approximate identity
if and only if it has bounded (right,left) approximate unity.
\end{corollary}
\begin{proposition}\label{3.2}
Let $\A$ be a Fr\'echet algebra and ${\U}$ be a countably
incomplete ultrapower. Then $({\A})_{\U}$ has (locally) bounded
approximate identity if and only if $\A$ dose have it too. The
same statement holds for left or right (locally) bounded
approximate identities.
\end{proposition}

\begin{proof}
Since Dixon's Theorem (see \cite[Prposition 2.9.3]{Dal}) also
holds for Fr\'echet algebras
 with respect to the above corollary, the proof is the same as the proof \ref{3.1} due to \cite[Proposition 2.2] {Daw}.
\end{proof}


\section{Diagonal for Fre\'chet algebras}\label{4}
Let $({\A},(R_n))$ and $({\mathcal{B}},(S_n))$ be two Fr\'echet
algebras. Then by \cite{kot} we can define the projective tensor
product ${\A}\hat{\otimes}\mathcal{B}$ with induced topology by
the following sub-multiplicative semi-norms
$$(R_n{\otimes}S_n)(T):=\inf\{\sum_{i=1}^{k}R_n(x_i)S_n(y_i):T=\sum_{i=1}^{k}{x_i{\otimes}y_i}\in {\A}{\otimes}\mathcal{B}\}.$$
And its completion given by
$$(R_n\hat{\otimes}S_n)(T)=\inf\{\sum_{i=1}^{\infty}R_n(x_i)S_n(y_i):T=\sum_{i=1}^{\infty}{x_i{\otimes}y_i}\in {\A}\hat{\otimes}\mathcal{B}\}.$$

The projective tensor product
 ${\A}\hat{\otimes}{\A}$ is a Fr\'echet $\A$-bimodule where the multiplication is specified by
$$a\cdot (b\otimes c) := ab\otimes c \quad and\quad (b\otimes c)\cdot a := b\otimes ca\quad \quad (a, b, c \in {\A}).$$

By above discussion we can define $(({\A}\hat{\otimes}{\mathcal
B})_{\U},(R_n\hat{\otimes}S_n)')$, where $(R_n\hat{\otimes}S_n)'$
is the semi-norm that we obtain from $R_n\hat{\otimes}S_n$ as
constructed of equation \ref{usd} in section 2, so
$({\A}\hat{\otimes}{\mathcal B})_{\U}$ is Fr\'echet algebra and
by Proposition \ref {2.2} we have:
$$(R_n\hat{\otimes}S_n)'(T_i)_{\U}=\lim_{\U}R_n\hat{\otimes}S_n(T_i),\quad\quad (T_i)_{\U}\in ({\A}\hat{\otimes}{\mathcal B})_{\U}.$$
\noindent {\bf Note.} For Fr\'echet algebra $({\A},(\p))$ we just
write $\ph$ instead of $\p\hat{\otimes}\p$ for
${\A}\hat{\otimes}{\A}$.\\

Let $(E,(R_n))$ and $(F,(S_n))$ be two Fr\'echet spaces. Then we
consider the canonical map
\begin{equation*}\label{4.3.1}
  \psi_0:(E)_{\U}\hat{\otimes}(F)_{\U}\rightarrow (E\hat{\otimes}F)_{\U},
\end{equation*}

 as in \cite{Daw} by
  $\psi_0((x_i)_{\U}\otimes (y_i)_{\U}):=(x_i\otimes y_i)_{\U}$, where $(x_i)_{\U}\in (E)_{\U}$, $(y_i)_{\U}\in(F)_{\U}$ and
   $(x_i\otimes y_i)_{\U}\in (E\hat{\otimes} F)_{\U}$.\\
If we replaced $E$ and $F$ in the above by   Fr\'echet algebra
$\A$, then $\psi_0$ is both ${\A}$-bimodule homomorphism and
$({\A})_{\U}$-bimodule homomorphism.
\begin{definition}
A locally convex space $E$ has approximation property whenever the
space of finite rank
 operators on $E$ is dense in the space of continuous operators on $E$ with respect to
 the topology of uniform convergence on all precompact subset of $E$.
\end{definition}
{\bf Note}. If $E$ is quasi-complete space then the topology of
uniform convergence on precompact subsets coincides on topology of
uniform convergence on compact subsets, especially in Fr\'echet
spaces.
\begin{lemma}
A locally convex space $E$ with the family of semi-norms
$(P_\alpha)_{\alpha\in I}$  has approximation property if and only
if for every locally convex space $F$ with the family of semi-
norms $(Q_\beta)_{\beta\in J}$, every continuous operator
$T:E\rightarrow F$, every precompact subset K of E and every
$\epsilon>0$, there exits a finite rank operator $S:E\rightarrow
F$ such that $Q_\beta(Tx-Sx)<\epsilon$ for every $x\in K$ and
$\beta \in J$.
\end{lemma}
\begin{proof}
Let $E$ has approximation property and $T:E\rightarrow F$ be
 a continuous operator. Given $\epsilon>0$ and precompact subset $K$
of $E$, there is a finite rank operator $U$ such that
$$P_\alpha(x-Ux)<\epsilon,
 \quad\quad  \alpha\in I, x\in K.$$
Since $T$ is continuous, there exists semi norms $\{P_i :
i=1,\cdots,n \}$ and a constant $C>0$, such that
 $$Q_\beta(Tx)\leq C{\sup}_{1\leq i\leq n} P_i(x),\quad\quad  \beta \in J, x\in E,$$ see ( \cite[Chapter 3, \S 1.1]{Shf}).
Now let $TU=S$. Then $S$ is finite rank operator from $E$ into
$F$ such that
 $$Q_\beta(Tx-TSx)<C{\sup}_{1\leq i\leq n}P_i(x-Sx)<C \epsilon, $$ for all $\beta\in J$. So it completes the proof.
The reverse is obvious.
\end{proof}

\begin{proposition}\label{appx}
Let $(E,(P_n))$ be a Fr\'echet space with approximation property
and let $(F,(Q_n))$ be a Fr\'echet space.  Suppose that for every
$\psi\in F'$ and for all $u=\sum_{n=1}^\infty x_n\otimes y_n\in
E\hat{\otimes} F$ where $(x_n)$ and $(y_n)$ are bounded sequences
in $E$ and $F$ respectively, if $\sum_{n=1}^\infty
\psi(y_n)x_n=0$ then $u=0$.
\end{proposition}

\begin{proof} Let $u\in E\hat{\otimes} F$. Recall that
$u$ can be written as a series $u=\sum_{n=1}^\infty
\lambda_n(x_n\otimes y_n)$, where $(x_n)$, $(y_n)$ are bounded
null sequences in $E$ and $F$ respectively and scalar sequence
$\{\lambda_n\}$ with $\sum_{n=1}^{\infty}|\lambda_n|<1$ \cite[$\S
41\ 4(6)$]{kot}.
 Let  $F'_b$ and ${\mathcal L}_b(E,F'_b)$ be respectively
the spaces of linear continuous operators from $F$ into $\Bbb C$
and $E$ into $F'_b$ with respect to the bounded convergence
topology (see \cite[\S 39]{kot}). Given $T\in {\mathcal
L}_b(E,F'_b)$ and $\epsilon>0$, since $E$ has the approximation
property, for each $\varepsilon >0$, there exists a finite rank
operator $S:E\rightarrow F'_b$,  such that
 $$R_M(Tx-Sx)<\epsilon, \quad\quad x\in K=\{x_n:n\in \Bbb{N}\}\cup\{0\},$$

  where $R_M$ is the semi norm related to locally
convex space $F'_b$,  is defined by $R_M(f):=\sup_{y\in M}|f(y)|$
for all $f\in F'_b$ and  bounded subset $M$  of $F$.
 We have
$Sx=\sum_{i=1}^m\phi_i(x)\psi_i$ where $\phi_i\in E'$ and
$\psi_i\in F'$. Let
$$\langle
u,S\rangle=\langle\sum_{n=1}^\infty x_n\otimes
y_n,S\rangle:=\sum_{n=1}^\infty(Sx_n)y_n.$$ Since by
\cite[Proposition 2]{Bir} ${\mathcal L}_b(E,F'_b)=(E\hat{\otimes}
F)'$ where the right hand side has the topology of uniform
convergence on the bounded sets of the form
$\bar{\Gamma}(C\otimes D)$, such that  $C$ and $D$ are bounded
subsets in $E$ and $F$ respectively.
In this case, we have
$$\langle u,S\rangle=\sum_{n=1}^\infty\sum_{i=1}^m \phi_i(x_n)\psi_i(y_n)=\sum_{i=1}^m\phi_i(\sum_{n=1}^\infty\psi_i(y_n)x_n)=0.$$
Therefore for bounded subset $M=\{y_n\}$ in $F$ we have
\begin{eqnarray*}
  |\langle u,T\rangle|&\leq& |\langle u,T-S\rangle|+|\langle u,S\rangle|\\
  & \leq &\sum_{n=1}^\infty |((T-S)\lambda_nx_n)y_n|  \\
  & \leq &\sum_{n=1}^\infty |\lambda_n|R_M(Tx_n-Sx_n) \\
  &\leq &\sum_{n=1}^\infty |\lambda_n|\epsilon \leq \epsilon.
\end{eqnarray*}
Since T and $\epsilon$ are arbitrary, it follows that $\langle
u,T\rangle=0$, thus $u=~0$, which is the desired conclusion.
\end{proof}


\begin{proposition}\label{inj}
If $(E)_{\U}$ with $\aleph(E)^+$-good ultrafilter has the approximation property, then $\psi_0$ is an
injection for any Fr\'echet space $F$.
\end{proposition}

\begin{proof} Suppose that  $T\in ((E)_{\U}\hat{\otimes}
(F)_{\U},(P'_n\hat{\otimes} Q'_n))$ has a representation
 $T=\sum_{n=1}^\infty x_n\otimes y_n$ with $\sum_{n=1}^\infty P'(x_n)Q'(y_n)<\infty$.
 If $(E)_{\U}$ has the approximation property then, by \ref{appx}, if $T\in ((E)_{\U}\hat{\otimes}
(F)_{\U},(P'_n\hat{\otimes} Q'_n))$ is
non-zero, then there exist $\mu\in (E)_{\U}'$ and $\lambda\in (F)_{\U}'$ with
$$0\not=\langle \mu \otimes \lambda ,T \rangle=\sum_{n=1}^\infty \langle \mu,x_n\rangle \langle \lambda,y_n\rangle.$$
As we only care about the value of $\mu$ on the countable set
$\{x_n\}$, by \cite[Theorem 3.4]{Hen1} we may suppose that $\mu
\in (E')_{\U}$, and similarly, that $\lambda\in (F')_{\U}$, say
$\mu = (\mu_i)$ and $\lambda=(\lambda_i)$. Let representatives
$x_n = (x_n^{(i)})$ and $y_n = (y_n^{(i)})$. Then by absolute
convergence,
$$\langle \mu \otimes \lambda ,T \rangle=\lim_{\U}\sum_{n=1}^\infty \langle \mu_i,x_n^{(i)}
\rangle \langle \lambda_i,y_n^{(i)}\rangle=\langle (\mu_i \otimes
\lambda_i) ,\psi_0(T) \rangle.$$
Hence we must have $\psi_0(T)\not =0$ and it completes the proof.
\end{proof}
\begin{proposition}\label{closedimage}
If $(E)_{\U}$ with $\aleph(E)^+$-good ultrafilter has the approximation property, then $\psi_0$ has closed image.
\end{proposition}
\begin{proof}

Let $0\not = u\in (E)_{\U}\hat{\otimes} (F)_{\U}$ and suppose
 $\sum_{n=1}^\infty \alpha_n(x_n\otimes
y_n)$ is a representation of $u$ with $x_n\rightarrow 0$,
$y_n\rightarrow 0$ and $\sum_{n=1}^{\infty}|\alpha_n|\leq 1$, see
\cite[\S 41\ 4(6)]{kot}.
Let $k=\sum_{n=1}^{\infty}|\alpha_n|$, since by \cite[Proposition
22.14] {Vog} for absolutely convex zero neighborhood $V$ of
$(E)_{\U}\hat{\otimes} (F)_{\U}$ we have
$||u||_V=\displaystyle{\sup_{T\in V^\circ}}|\langle u,T
\rangle|$, without loss of generality we can assume
$0\not=||u||_V$ (otherwise we can find a $V$ such that
$0\not=||u||_V$) so for $\epsilon=\frac{||u||_V}{2(k+1)}$ there
exists a $T\in V^\circ$ such that
$\frac{||u||_V(2k+1)}{2(k+1)}=||u||_V-\frac{||u||_V}{2(k+1)}\leq
\langle u , T \rangle$.
Now because ${\mathcal
L}_b((E)_{\U},((F)_{\U})'_b)=((E)_{\U}\hat{\otimes}(F)_{\U})'$, by
\cite[Proposition 2]{Bir} we can assume $T\in {\mathcal
L}_b((E)_{\U},((F)_{\U})'_b)$ and since $(E)_{\U}$ has
approximation property there exists a finite rank operator
$S:(E)_{\U}\longrightarrow (F)_{\U}'$ such that
$R_M(Tx-Sx)<\frac{||u||_V}{2(k+1)}$ for $x\in\{x_n\}\cup\{0\}$
and $M=\{y_n\}$ where $R_M$ is as in Proposition \ref{appx}.
So we have
\begin{eqnarray*}
\langle u , T-S \rangle &\leq& \sum _{n=1}^\infty
|((T-S)\alpha_nx_n)y_n|\\
 & \leq & \sum_{n=1}^\infty |\alpha_n|R_M((T-S)x_n)\\
& \leq & \epsilon \sum_{n=1}^\infty |\alpha_n|.
 \end{eqnarray*}
and it follows that:
\begin{eqnarray*}
\langle u, S \rangle &\geq &\langle u, T \rangle-\epsilon \sum_{n=1}^\infty |\alpha_n| \\
&\geq &
\frac{(2k+1)||u||_V}{2(k+1)}-\frac{k||u||_V}{2(k+1)}=\frac{||u||_V}{2}.
\end{eqnarray*}
Let $\sum_{j=1}^N\mu_j\otimes \lambda_j$ representation of $S$
with $(\mu_j)\subset (E)_{\U}'$ and $(\lambda_j)\subset (F)_{\U}'$
so that $\langle u, S \rangle =\sum_{n=1}^\infty \sum_{j=1}^N
\langle \mu_j , \alpha_nx_n \rangle \langle \lambda_j , y_n
\rangle $. By consideration of closed span $\{x_n\}$ and
\cite[Theorem 3.4]{Hen1}, we can assume that $\mu_j\in (E')_{\U}$
and $\lambda_j\in (F')_{\U}$ for each j.
Now let representations  $\mu_j=(\mu_j^i)$,
$\lambda_j=(\lambda_j^i)$ and for each $i$ let $S_i=\sum_{j=1}^N
\mu_j^i\otimes \lambda_j^i$.
By a calculation we have $\langle u, S \rangle = \langle (S_i) ,
\psi_0 (u) \rangle$. So:
$$|\langle (S_i) , \psi_0 (u)|\geq \frac{||u||_V}{2}.$$
This implies that if $\{\psi_0(u_n) \}$ is cauchy sequence then
$\{u_n\}$ is cauchy sequence and since
$(E)_{\U}\hat{\otimes}(F)_{\U}$ is complete it converges to an
element $u$ in  $(E)_{\U}\hat{\otimes}(F)_{\U}$. Then $\lim_n
\psi_0(u_n)=\psi(u)$ and hence this shows that the image of
$\psi_0$ is closed.
\end{proof}
\begin{proposition}\label{bb}
Let $E$, $F$ be Fr\'echet spaces and $(E)_{\U}$ with
$\aleph(E)^+$-good ultrafilter has approximation property. Then
$$\psi_0:((E)_{\U}\hat{\otimes} (F)_{\U},(P'_n\hat{\otimes}
Q'_n))\rightarrow ((E\hat{\otimes}
F)_{\U},((P_n\hat{\otimes}Q_n)'))$$ is continuous and bounded
below. Moreover, there exists $(P_{n_0}\hat{\otimes}Q_{n_0})'$
such that for all $V\in (E)_{\U}\hat{\otimes}(F)_{\U}$ and $n\in
\Bbb{N}$, $P'_n\hat{\otimes}Q'_n(V)\leq C_n
(P_{n_0}\hat{\otimes}Q_{n_0})'(\psi_0(V))$ and
$(P_n\hat{\otimes}Q_n)'(\psi_0(V))\leq P'_n\hat{\otimes}Q'_n(V)$,
for some $C_n>0$.
\end{proposition}
\begin{proof} Let $V\in (E)_{\U}\otimes (F)_{\U}$ and $n\in \Bbb{N}$. Then

\begin{eqnarray*}
P'_n\otimes Q'_n(V)&=&\inf\{\sum_{j=1}^k P'_n(x^j)Q'_n(y^j):V= \sum_{j=1}^k (x_i^j)_{\U}\otimes (y_i^j)_{\U}\}\\
&=&\inf\{\sum_{j=1}^k \lim_{{\U}}P_n(x_i^j)\lim_{{\U}}Q_n(y_i^j):V=\sum_{j=1}^k (x_i^j)_U\otimes (y_i^j)_{\U}\}\\
&=&\inf\{\sum_{j=1}^k \lim_{\U}P_n(x_i^j)Q_n(y_i^j):V=\sum_{j=1}^k (x_i^j)_{\U}\otimes (y_i^j)_{\U}\}\\
&=&\inf\{\lim_{\U}\sum_{j=1}^k P_n(x_i^j)Q_n(y_i^j):V=\sum_{j=1}^k (x_i^j)_{\U}\otimes (y_i^j)_{\U}\}\\
&=&\inf\{\lim_{\U}\sum_{j=1}^k P_n(x_i^j)Q_n(y_i^j):V'=\sum_{j=1}^k (x_i^j\otimes y_i^j)_{\U}\}\\
&=&\inf\{\lim_{\U}\sum_{j=1}^k
P_n(x_i^j)Q_n(y_i^j):V'=(\sum_{j=1}^k x_i^j\otimes y_i^j)_{\U}\}
\end{eqnarray*}

Let $V'_i=\sum_{j=1}^k x_i^j\otimes y_i^j$. Then  $\sum_{j=1}^k
P_n(x_i^j)Q_n(y_i^j)\geq P_n\otimes Q_n(V'_i)$, and hence
$$\lim_{\U}\sum_{j=1}^k P_n(x_i^j)Q_n(y_i^j)\geq\lim_{\U} P_n\otimes Q_n(V'_i)=(P_n\otimes Q_n)'(V'),$$
so $P'_n\otimes Q'_n(V)\geq (P_n\otimes Q_n)'(V')$. Since the
semi norms $P'_n\otimes Q'_n$ and $(P_n\otimes Q_n)'$ are
continuous on $(E)_{\U}\otimes (F)_{\U}$ and $(E\otimes F)_{\U}$
respectively, then
$$(P_n\hat{\otimes} Q_n)'(V')\leq P'_n\hat{\otimes}
Q'_n(V), \quad\quad V\in ((E)_{\U}\hat{\otimes} (F)_{\U}), V'\in
(E\hat{\otimes} F)_{\U}.$$
Suppose that $V'=\psi_0(V)$, from the last inequality we have,
$$(P_n\otimes Q_n)'(\psi_0(V))\leq P'_n\otimes Q'_n(V),$$
so this shows that $\psi_0$ is continuous.
Since $\psi_0$ by the proposition \ref{closedimage} has closed
image and by the proposition \ref{inj} is injective, from the
\cite[Corollary 2.12]{Rud} we have $\psi_0^{-1}$ is continuous.

Moreover for every  $P'_n\hat{\otimes} Q'_n$ by \cite[Theorem
1.1,\S 3]{Shf} there exists a constant $C_n>0$ and
$(P_{n_0}\hat{\otimes} Q_{n_0})'$ such that
 $P'_n\hat{\otimes}Q'_n(\psi_0^{-1}(V))\leq C_n (P_{n_0}\hat{\otimes}Q_{n_0})'(V)$ and
$P'_n\hat{\otimes}Q'_n(V)\leq C_n
(P_{n_0}\hat{\otimes}Q_{n_0})'(\psi_0(V)$. So this completes the
proof.
\end{proof}
\begin{proposition}\label{4.3}
Let $(E,(P_n))$, and $(F,(Q_n))$ be Fr\'echet spaces, and $\U$ be
an ultrafilter on an index set $I$ and let $T\in
(E\hat{\otimes}F)_{\U}$. Then the following statements are
equivalent:
\begin{enumerate}
  \item[{ (i)}] There exist a sequence $\alpha_{k,n}$ of positive real numbers
  such that
$\sum_{k,n}\alpha_{k,n}<\infty$ and  $T=(T_i)$ admits a
representation of the form

$$T_i=\sum_{k=1}^{\infty}x_k^{(i)}\otimes y_k^{(i)}\in E\hat{\otimes}F,\quad\quad i\in I,$$

where for each $i, k$ we have $\p(x_k^{(i)})Q_n(y_k^{(i)})\leq
\alpha_{k,n}.$
  \item[{ (ii)}]  $T$ lies in the image of $\psi_0$.
  \end{enumerate}
\end{proposition}
\begin{proof}
Suppose that (i) holds. By rescaling, we may suppose that
$\p(x_k^{(i)})=Q_n(y_k^{(i)})\leq \alpha_{k,n}^{1/2}$ for each
$i\in I$ and $k\geq 1$. For each $k\geq1$, let
$$x_k=(x_k^{(i)})\in (E)_{\U},\ \ y_k=(y_k^{(i)})\in (F)_{\U},$$
with $P'_n(x_k)\leq \alpha_{k,n}^{1/2}$ and $Q'_n(y_k)\leq
\alpha_{k,n}^{1/2}$. Now let
$$\sigma=\sum_{k=1}^{\infty}x_k\otimes y_k\in (E)_{\U}\hat{\otimes}(F)_{\U},$$
and  $\sigma_m=\sum_{k=1}^{\infty}x_k\otimes y_k$, so
$\sigma_m\rightarrow \sigma \in(E)_{\U}\hat{\otimes}(F)_{\U}$.
Then
\begin{eqnarray*}
  \lim_{m\to \infty}(P_n\hat{\otimes}Q_n)'(\psi_0(\sigma_m)-T) & = &\lim_{m\to \infty}\lim_{\U}(P_n\hat{\otimes}Q_n)(\sum_{k=1}^{m}x_k^{(i)}\otimes y_k^{(i)}-T_i) \\
   & \leq &\lim_{m\to \infty}\lim_{\U}\sum_{k=m+1}^{\infty}P_n(x_k^{(i)})Q_n(y_k^{(i)}) \\
   & \leq &\lim_{m\to \infty}\sum_{k=m+1}^{\infty}\alpha_{k,n}=0,
\end{eqnarray*}
so this shows that $\psi_0(\sigma)=T$, as required. Conversely,
suppose that $\psi_0(\sigma)=T$ for

$$\sigma=\sum_{k=1}^{\infty}x_k\otimes y_k\in (E)_{\U}\hat{\otimes}(F)_{\U},$$
with $\sum_{k=1}^{\infty}P'_n(x_k)Q'_n(y_k)<\infty$. And let
$\alpha_{k,n}=P'_n(x_k)Q'_n(y_k)$ and  pick representatives
$x_k=(x_k^{(i)})\in (E)_{\U}\ \ and \ \ y_k=(y_k^{(i)})\in
(F)_{\U},$ such that $P'_n(x_k)=P(x_k^i)$ and $Q'_n(x_k)=Q(x_k^i)$
for each $i\in I$. If we assume
$$T_i=\sum_{k=1}^{\infty}x_k^{(i)}\otimes y_k^{(i)},\quad \sigma_m=\sum_{k=1}^m x_k\otimes y_k,$$

then  $T=\lim_{m\to \infty}\psi_0(\sigma_m)$. For each $m\in
\mathbb{N}$, we have
\begin{eqnarray*}
(P_n\hat{\otimes}Q_n)'((T_i-\psi_0(\sigma_m))&=&\lim_{\U}(P_n\hat{\otimes}Q_n)(T_i-\sum_{k=1}^{m}x_k^{(i)}\otimes y_k^{(i)})\\
&\leq &\lim_{\U}\sum_{k=m+1}^{\infty}\p(x_k^{(i)})\p(y_k^{(i)})\\
&=&\sum_{k=m+1}^{\infty}\alpha_{k,n}
\end{eqnarray*}

and so taking limits on both sides of the above equation as
$m\rightarrow \infty$, we see that $(T_i)_{\U}=T$, as required to
complete the proof.
\end{proof}

\section{Contractibility of  ultrapower of Fr\'echet algebras}
If $\A$ is a Fr\'echetf algebra, then the corresponding diagonal
operator is defined through,
$\Delta_{\A}:{\A}\hat{\otimes}{\A}\rightarrow {\A}$ by
$\Delta_{\A}(a\otimes b)=ab$ for all $a,b\in {\A}$. By definition
a diagonal for $\A$ is an element $T\in {\A}\hat{\otimes}{\A}$
such that $\Delta_{\A}(T)a=a$ and $a\cdot T=T\cdot a$ $(a\in\A)$,
for more details see \cite{Run}. We now introduce a 'locally'
version of diagonal for Fr\'echet algebras.

\begin{definition}\label{4.1}
Let $(\A,(\p))$ be a Fr\'echet algebra. We say that $\A$ has a locally
diagonal if for each $n\in \Bbb{N}$, there exists  $T_n\in
{\A}\hat{\otimes}{\A}$ such that $P_n(\Delta_{\A}(T_n)a-a)=0$ and
$P_n(a\cdot T_n-T_n\cdot a)=0$ for all $a\in\A$.
\end{definition}

\begin{definition}\label{4.1.1}
Let $\A$ be Fr\'echet algebra, we call $\A$ contractible if it has diagonal.
\end{definition}

\begin{theorem}\label{4.4}
Let $(\A,(\p))$ be a Fr\'echet algebra and $\U$ be an
$\aleph({\A})^+$-good ultrafilter on an index set $I$, such that
$({\A})_{\U}$ has approximation property. Then $({\A})_{\U}$ has
(locally) diagonal if and only if for  a constant
  $C>0$, there exist a sequence of positive reals
    $(\alpha_{j,n})_{n,j\in \Bbb N}$ with $\displaystyle{\sum_{j,n}\alpha_{j,n}}\leq C$,
    such that for each $i\in I$ there exists sequences $(a^i_{j})_{j\in \Bbb N}$ and $(b^i_{j})_{j\in \Bbb
    N}$ in $\A$ with following condition:
\begin{itemize}
  \item $T_i=\displaystyle{\sum_{j=1}^\infty b_j^i\otimes a_j^i\in{\A}\hat{\otimes}{\A}}$,
  \item $\lim_{\U}\ph(a_i\cdot T_i-T_i\cdot a_i)=0$, \quad\quad for all $a=(a_i)\in ({\A})_{\U}$
  \item $\lim_{\U}\p(\Delta_{\A}(T_i)a_i-a_i)=0$, \quad\quad for all $a=(a_i)\in ({\A})_{\U}$
  \item $\p(b_j^i)\p(a_j^i)\leq \alpha_{j,n}$, \quad\quad for all $n,j\in
  \Bbb{N}$.
\end{itemize}
\end{theorem}

\begin{proof} Let $\{T_n\}$ be a (locally) diagonal in $({\A})_{\U}$. By definition of $\psi_0$ the
following diagram commutes.

\begin{center}

\begin{picture}(0,0)
\setlength{\unitlength}{1in}
\put(-.5,0){{${(\A)}_{\U}\hat{\otimes}({\A})_{\U}$}}
\put(.3,-.02){\vector(1,0){0.55}}
\put(.38,.06){$\Delta_{({\A})_{\U}}$} \put(1,0){${(\A)}_{\U}$}
\put(0,-.1){\vector(0,-1){.55}} \put(0.05,-.4){$\psi_0$}
\put(-.3,-.8){$({\A}\hat{\otimes}{\A})_{\U}$}
\put(.2,-.65){\vector(4,3){.75}}
\put(.48,-.55){$(\Delta_{\A})_{\U}$}
\end{picture}

\end{center}\vspace{1in}
Since by the proposition \ref{inj} $\psi_0$ is injective, then
$\psi_0(T_n)=(T_i^n)_{i\in I}$ is not zero. So for all
$a=(a_i)\in(\A)_{\U}$ we have

\begin{eqnarray*}
  \lim_{\U}\hat{P}_n(a_i\cdot T_i^n-T_i^n\cdot a_i) &=& (P_{n}\hat{\otimes}P_{n})'(a\cdot (T_i^n)_{i\in I}-(T_i^n)_{i\in
I}\cdot a)\\
&=&(P_{n}\hat{\otimes}P_{n})'(a\cdot \psi_0(T_n)-\psi_0(T_n)\cdot
a)\\&=&(P_{n}\hat{\otimes}P_{n})'(\psi_0(a\cdot T_n-T_n\cdot a))=0,
\end{eqnarray*}

with same method we have
$$\lim_{\U}\p(\Delta_{\A}(T_i^n)a_i-a_i)=0.$$
From the proposition \ref{4.3}, the above condition holds.

For the reverse, since ${(\A)}_{\U}$ has approximation property,
then $\psi_0$ is bounded below and  the proposition \ref{bb}
implies that,
  for all $n\in \Bbb{N}$, there is constant $M_n>0$ and $n_0\in \Bbb{N}$, such that
  $M_n(P'_{n}\hat{\otimes}P'_{n})(a)\leq (P_{n_0}\hat{\otimes}P_{n_0})'(\psi_0(a))$ for all
   $a\in {(A)}_{\U}\hat{\otimes}{(\A)}_{\U}$. Now for all $i\in I$, let $T_n=(T_i^n)$ in the above
   condition so
    by proposition \ref{4.3} $T_n$ lies in the
image of $\psi_0$. And hence there is
     $S_n\in {(A)}_{\U}\hat{\otimes}{(\A)}_{\U}$ such that $\psi_0(S_n)=T_n$.

     Now for $a\in {(A)}_{\U}\hat{\otimes}{(\A)}_{\U}$ and $n \in {\Bbb N}$ we have:

\begin{enumerate}
  \item[({i})]     $(P'_{n}\hat{\otimes}P'_{n})(a\cdot S_n-S_n\cdot
  a)=0$,   because
\begin{eqnarray*}
  (P'_{n}\hat{\otimes}P'_{n})(a\cdot S_n-S_n\cdot a)& \leq &\frac{1}{M_n}(P_{n_0}\hat{\otimes}P_{n_0})'(\psi_0(a\cdot S_n-S_n\cdot a))
  \\
  &=&\frac{1}{M_n}(P_{n_0}\hat{\otimes}P_{n_0})'(a\cdot T_n-T_n\cdot a)=0
  \end{eqnarray*}

  \item[({ii})]  $P_n(\Delta_{{(\A)}_{\U}}(S_n) a-a)=0$, because
   \begin{eqnarray*}
  (P'_{n}\hat{\otimes}P'_{n})(\Delta_{{(\A)}_{\U}}(S_n) a-a)&=&(P'_{n}\hat{\otimes}P'_{n})((\Delta_{\A})_{\U}(\psi_0(S_n))a-a)\\
  & =&(P'_{n}\hat{\otimes}P'_{n})((\Delta_{\A})_{\U}(T_n)a-a)=0.
    \end{eqnarray*}
\end{enumerate}
This shows that $\{S_n\}$ is a (locally) diagonal in ${(\A)}_{\U}$.
\end{proof}
\begin{corollary}\label{4.5}
Let $\A$ be a Fr\'echet algebra with $\aleph({\A})^+$-good ultrafilter such that has approximation property then
${(\A)_{\U}}$ has (locally) diagonal if and only if $\A$ has (locally) diagonal.
\end{corollary}

\begin{proof} Since every ultrafilter is equal to a net
(see \cite {kot}) it follows from above theorem.
\end{proof}

\begin{theorem}\label{4.9}
Let $\A$ be a Fr\'echet algebra with approximation property and let $\U$ be a $\aleph({\A})^+$-good ultrafilter.
${(\A)_{\U}}$ is contractible if and only if $\A$ is contractible.
\end{theorem}

\begin{example}
Let $A=(a_{i,j})_{i,j\in {\Bbb{N}}}$ be a k\"{o}the matrix which is
non-negative real matrix that satisfies the following condition:
\begin{enumerate}
  \item for each $i\in {\Bbb{N}}$ there exists a $j\in {\Bbb{N}}$ with
  $a_{i,j}>0$
  \item $a_{i,j}\leq a_{i,j+1}$ for all $i,j\in {\Bbb{N}}$.
\end{enumerate}
For each k\"{o}the matrix we define:
$$\lambda^1(A)=\{x=(x_i)\in {\Bbb{C}}^{\Bbb{N}}: P_k(x)=\sum_{i=1}^\infty |x_i|a_{i,k}<\infty, \quad k\in {\Bbb{N}}\}.$$
$\lambda^1(A)$ is Fr\'echet space with fundamental family of
semi-norms $(P_k)_{k\in {\Bbb{N}}}$. This space with additional
condition $a_{i,j}\in \{0\}\cup [1,{\infty})$ for all $i,j\in
{\Bbb{N}}$ and pointwise multiplication is Fr\'echet algebra.

In the above description, we take $a_{i,j}=1$ for all $i\leq j$ and zero for
other. It is easy to show that
$\lambda^1(A)={\Bbb{C}}^{\Bbb{N}}$ with following semi-norms
$$P_k(x)=\sum_{i=1}^\infty |x_i|\quad\quad x\in {\Bbb{C}}^{\Bbb{N}}.$$

Since ${\Bbb{C}^{\Bbb{N}}}= \displaystyle{\lim_{\longleftarrow}}
{{\Bbb{C}}^{n}}$ is reduced projective limit (see \cite{Shf}) and
$({\Bbb{C}}^n)_{\U}={\Bbb{C}}^n$, for all $n\in {\Bbb{N}}$, so with a good ultrafilter by
\cite[Corollary 2.3]{Hen1} we have
$$({\Bbb{C}^{\Bbb{N}}})_{\U}=
\displaystyle{\lim_{\longleftarrow}}
({{\Bbb{C}}^{n}})_{\U}=\displaystyle{\lim_{\longleftarrow}}
{{\Bbb{C}}^{n}}={\Bbb{C}}^{\Bbb{N}}.$$
Now by \cite[Corollary 2,chapter
III] {Shf} ${\Bbb{C}}^{\Bbb{N}}$ has approximation property and
since  ${\Bbb{C}}^{\Bbb{N}}$ is  contractible, then its ultrapower $({\Bbb{C}^{\Bbb{N}}})_{\U}$ is contractible.
\end{example}


\section{Acknowledgments}\label{5}
The authors would like to thank  to the help of the prof. B. Sims
for sending us his interesting book ```` Ultra"-techniques in
Banach space theory''.



\end{document}